\documentclass[12pt, reqno]{amsart}
\usepackage{amsmath, amsthm, amscd, amsfonts, amssymb, graphicx, color}
\usepackage[bookmarksnumbered, colorlinks, plainpages]{hyperref}
\hypersetup{colorlinks=true,linkcolor=red, anchorcolor=green, citecolor=cyan, urlcolor=red, filecolor=magenta, pdftoolbar=true}
\input{mathrsfs.sty}

\newtheorem{theorem}{Theorem}[section]

\newtheorem{proposition}[theorem]{Proposition}
\newtheorem{cor}[theorem]{Corollary}
\theoremstyle{definition}
\newtheorem{definition}[theorem]{Definition}

\theoremstyle{remark}
\newtheorem{remark}[theorem]{\bf{Remark}}
\numberwithin{equation}{section}
\begin{document}

\title [Generalized Euclidean operator radius inequalities of a pair operators] {\small {Generalized Euclidean operator radius inequalities of a pair of bounded linear operators }}

\author[S. Jana ] {Suvendu Jana}

\address{$^1$ Department of Mathematics, Mahisadal Girls' College, Purba Medinipur 721628, West Bengal, India}
\email{janasuva8@gmail.com}

\renewcommand{\subjclassname}{\textup{2020} Mathematics Subject Classification}

\maketitle

\begin{abstract}
	Let $ \mathbb{B}(\mathscr{H})$ represent the $C^*$-algebra, which consists of all bounded linear operators on $\mathscr{H},$ and let $N ( .) $ be a norm on $ \mathbb{B}(\mathscr{H})$. We define a norm  $w_{(N,e)} (. , . )$ on $ \mathbb{B}^2(\mathscr{H})$ by $$ w_{(N,e)}(B,C)=\underset{|\lambda_1|^2+\lambda_2|^2\leq1}\sup \underset{\theta\in\mathbb{R}}\sup      N\left(\Re \left(e^{i\theta}(\lambda_1B+\lambda_2C)\right)\right),$$   for every $B,C\in\mathbb{B}(\mathscr{H})$ and  $\lambda_1,\lambda_2\in\mathbb{C}.$  We investigate basic properties of this norm and prove some bounds involving it. In particular, when $N( .)$ is the Hilbert-Schmidt norm, we prove some Hilbert-Schmidt Euclidean operator radius inequalities for a pair of bounded linear operators.
\end{abstract}

\section{Introduction}

Consider a complex Hilbert space $\mathscr{H}$ with inner product $\langle \cdot,\cdot \rangle $ and the corresponding norm $\|\cdot\|$. Let $ \mathbb{B}(\mathscr{H})$ represent the $C^*$-algebra, which consists of all bounded linear operators on $\mathscr{H},$ including the identity operator $I$.  For $T\in \mathbb{B}(\mathscr{H}),$ $T^*$ denotes the adjoint of $T$. The real part and imaginary part of $T$ denoted by  $ \Re (T)$ and $\Im(T),$    are defined as $\Re(T)=\frac{1}{2}(T+T^*)$ and $\Im(T)=\frac{1}{2\rm i}(T-T^*)$ respectively. The numerical range of $T$, denoted by $W(T)$, is defined as $W(T)=\left \{\langle Tx,x \rangle: x\in \mathscr{H}, \|x\|=1 \right \}.$
We denote by $\|T\|$, $ c(T) $ and $w(T)$ the operator norm, the Crawford number and the numerical radius of $T$, respectively. Note  that $$c(T)=\inf \left \{|\langle Tx,x \rangle|: x\in \mathscr{H}, \|x\|=1 \right \}$$ and $$w(T)=\sup \left \{|\langle Tx,x \rangle|: x\in \mathscr{H}, \|x\|=1 \right \}.$$ 
It is well known that the numerical radius $ w(\cdot)$ defines a norm on $\mathbb{B}(\mathscr{H})$ and is equivalent to the operator norm $\|\cdot\|$. In fact, the following double inequality holds:
\begin{eqnarray}\label{eqv}
\frac{1}{2} \|T\|\leq w({T})\leq\|T\|.
\end{eqnarray}
The inequalities in (\ref{eqv}) are sharp. The first inequality becomes equality if $T^2=0$, and the second one turns into equality if $T$ is normal. Kittaneh \cite{E} improved the inequalities in (\ref{eqv}) by establishing that 
\begin{eqnarray}\label{k5}
\frac{1}{4}\|T^*T+T{T}^*\|\leq w^2({T})\leq\frac{1}{2}\|T^*T+T{T}^*\|.
\label{d}\end{eqnarray}
For further improvements of \eqref{eqv} and \eqref{k5} we refer the interested readers to the recent book \cite{PDMK}. 
Let $B,C\in \mathbb{B}(\mathscr{H})$, the Euclidean operator radius of $B$ and $C$, denoted by $w_e(B,C),$ is defined as $$w_e(B,C)= \sup \left \{ \sqrt{|\langle B x,x\rangle|^2+|\langle C x,x\rangle|^2} : x\in \mathscr{H}, \|x\|=1 \right \}.$$ 
Following \cite{P}, $w_e(.\,,\,.): \mathbb{B}^2(\mathscr{H})\to [0,\infty]$ is a norm that satisfies the inequality
 \begin{eqnarray}
\frac{\sqrt{2}}{4}\|B^*B+C^*C\|^{\frac{1}{2}}\leq w_e(B,C)\leq\|B^*B+C^*C\|^{\frac{1}{2}}.
\label{eqn1}\end{eqnarray}
The constants $\frac{\sqrt{2}}{4}$ and $1$ are best possible in \eqref{eqn1}.
 If $B$ and $C$ are self-adjoint operators, then $(\ref{eqn1})$ becomes 
 \begin{eqnarray}
\frac{\sqrt{2}}{4}\|B^2+C^2\|^{\frac{1}{2}}\leq w_e(B,C)\leq\|B^2+C^2\|^{\frac{1}{2}}.
\label{eqn2}\end{eqnarray}
We note that for self-adjoint operators $B$ and $C$, $w_e(B,C)=w(B+i C),$ its proof follows easily from the definition of $w_e(B,C)$.
In \cite[Th. 1]{D}, Dragomir proved that if  $B,C\in \mathbb{B}(\mathscr{H})$, then
\begin{eqnarray}\label{D06}
\frac12 w(B^2+C^2)	\leq w^2_e(B,C)\leq\|B^*B+C^*C\|,
\end{eqnarray}
 where the constant $\frac12$ is best possible in the sense that it cannot be replaced by a larger constant.  For further extension of Euclidean operator radius and related inequalities we refer to \cite{s4,s2,s3,s1,SMS, SAH21}.
In \cite{STU07}, Yamazaki gave an important and useful identity for $w(T)$, 
$$ w(T)=\underset{\theta\in\mathbb{R}}\sup \|\Re(e^{i\theta}T)\|.$$
Motivated by the above characterization, Abu-Omar and Kittaneh \cite{AK2019} generalized the usual numerical radius as follows:
\begin{eqnarray}\label{eqn11}
 w_N(T)=\underset{\theta\in\mathbb{R}}\sup      N(\Re(e^{i\theta}T)),\end{eqnarray} where $N(.)$ is a norm on $\mathbb{B}(\mathscr{H})$.\\
They proved that $w_N(.)$ defines a norm $\mathbb{B}(\mathscr{H})$ and $w_N(.)$ is self adjoint i.e, $w_N(T)=w_N(T^*)$. Also proved that $w_N(T)\geq\frac12 N(T)$ and $w_N(T)\leq N(T)$, if $N(T^*)=N(T)$, for every $T\in\mathbb{B}(\mathscr{H})$.\\
Popescu \cite[Section 2, Corollary 2.3]{P}, gave a  characterization for Euclidean operator radius as follows: $$ w_e(B,C)= \underset{|\lambda_1|^2+\lambda_2|^2\leq1}\sup w(\lambda_1B+\lambda_2C),$$ where $\lambda_1,\lambda_2\in\mathbb{C}.$\\
For $2$-tuple operators $B=(B_1,B_2),$ $C=(C_1,C_2)\in\mathbb{B}^2(\mathscr{H})$, we write $B+C=(B_1+C_1, B_2+C_2),$ $ BC=(B_1C_1, B_2C_2),$ and $\alpha B=(\alpha B_1,\alpha C_1)$, for any scaler $\alpha\in\mathbb{C}.$\\

 In Section 2,  inspired by \cite{AK2019}, for an arbitrary norm $N( . )$ on $\mathbb{B}(\mathscr{H})$, we define the $w_{(N,e)}(.,.)$ as a generalization of the Euclidean operators radius of a pair of operator and investigate basic properties of this norm and prove inequalities involving it. In section 3,  we prove some bounds for Hilbert-Schmidt Euclidean operators radius for pair of bonded linear operators when $N ( . )$ is the Hilbert-Schmidt norm $\| . \|_2$.

\section{A generalization of Euclidean operator radius}
In this section, we introduce our new norm on $\mathbb{B}^2(\mathscr{H})$, which generalizes the
Euclidean operator radius and present basic properties of this norm. 

\begin{definition}\label{def1}
Let $N ( . )$ be a norm on $\mathbb{B}(\mathscr{H})$. The function $w_{(N,e)}( . , . ):\mathbb{B}^2(\mathscr{H})\rightarrow [0,\infty)$ is defined as: $$ w_{(N,e)}(B,C)=\underset{|\lambda_1|^2+\lambda_2|^2\leq1}\sup \underset{\theta\in\mathbb{R}}\sup     N\left(\Re \left(e^{i\theta}(\lambda_1B+\lambda_2C)\right)\right),$$ for every $B,C\in\mathbb{B}(\mathscr{H})$ and  $\lambda_1,\lambda_2\in\mathbb{C}.$ 

\end{definition}
 
In our next theorem, we prove that $w_{(N,e)}( . , . )$ is a norm on $\mathbb{B}^2(\mathscr{H})$ . We use
some ideas of \cite[Theorems 1]{AK2019}.

\begin{theorem}\label{th1}  $w_{(N,e)}(.  ,  . )$ is a norm on $\mathbb{B}^2(\mathscr{H})$.
	
\end{theorem}

\begin{proof}
Let $B,C\in \mathbb{B}(\mathscr{H})$. Since $N(.)$ is a norm on $\mathbb{B}(\mathscr{H})$, we have \\ $ N\left(\Re \left(e^{i\theta}(\lambda_1B+\lambda_2C)\right)\right)\geq 0$ for every $\theta\in\mathbb{R}$. Hence $\underset{|\lambda_1|^2+\lambda_2|^2\leq1}\sup \underset{\theta\in\mathbb{R}}\sup     N\left(\Re \left(e^{i\theta}(\lambda_1B+\lambda_2C)\right)\right)\geq0$. So  $w_{(N,e)}( B,C )\geq 0.$\\
Let us assume that $w_{(N,e)}( B,C )= 0$. Then $ N\left(\Re \left(e^{i\theta}(\lambda_1B+\lambda_2C)\right)\right)=0$, for all $\theta\in\mathbb{R}$ and  $\lambda_1,\lambda_2\in\mathbb{C}$ with $|\lambda_1|^2+\lambda_2|^2\leq1.$ Taking $\theta =0$ , $\theta=\frac{\pi}{2}$ , and  $\lambda_1=1$,  $\lambda_2=0$, we get $N(\Re B)=0=N(\Im B).$ So $B=0$. In Similar way, we obviously get $C=0$. \\
Let $\alpha \in\mathbb{C}$. There exists $\phi \in \mathbb{R}$ such that $\alpha = |\alpha|e^{i\phi}$. Hence \begin{eqnarray*}
w_{(N,e)}( \alpha B,\alpha C )&=&\underset{|\lambda_1|^2+\lambda_2|^2\leq1}\sup \underset{\theta\in\mathbb{R}}\sup      N(\Re( e^{i\theta}(\lambda_1\alpha B+\lambda_2 \alpha C)))\\
&=&\underset{|\lambda_1|^2+\lambda_2|^2\leq1}\sup \underset{\theta\in\mathbb{R}}\sup      N(\Re( e^{i(\theta+\phi)}(\lambda_1|\alpha| B+\lambda_2 |\alpha| C)))\\
&=&\underset{|\lambda_1|^2+\lambda_2|^2\leq1}\sup \underset{t  \in\mathbb{R}}\sup      N(\Re (e^{it}(\lambda_1|\alpha| B+\lambda_2 |\alpha| C)))\\
&=&|\alpha|\underset{|\lambda_1|^2+\lambda_2|^2\leq1}\sup \underset{t  \in\mathbb{R}}\sup      N(\Re( e^{it}(\lambda_1 B+\lambda_2  C)))\\
&=&|\alpha|w_{(N,e)}(  B, C ).
\end{eqnarray*} 
Let $B_1,C_1,B_2,C_2\in\mathbb{B}(\mathscr{H}).$\\
\begin{eqnarray*}
&&w_{(N,e)}(  B_1+B_2, C_1+C_2 )\\&=&\underset{|\lambda_1|^2+\lambda_2|^2\leq1}\sup \underset{\theta\in\mathbb{R}}\sup      N(\Re( e^{i\theta}(\lambda_1( B_1+B_2)+\lambda_2 ( C_1+C_2))))\\
&=&\underset{|\lambda_1|^2+\lambda_2|^2\leq1}\sup \underset{\theta\in\mathbb{R}}\sup      N(\Re( e^{i\theta}((\lambda_1 B_1+\lambda_2  C_1)+(\lambda_1 B_2+\lambda_2  C_2))))\\
&\leq & \underset{|\lambda_1|^2+\lambda_2|^2\leq1}\sup \underset{\theta\in\mathbb{R}}\sup      N(\Re( e^{i\theta}(\lambda_1 B_1+\lambda_2  C_1)))\\
&&+\underset{|\lambda_1|^2+\lambda_2|^2\leq1}\sup \underset{\theta\in\mathbb{R}}\sup      N(\Re( e^{i\theta}(\lambda_1 B_2+\lambda_2  C_2)))\\
&=& w_{(N,e)}(  B_1, C_1 )+w_{(N,e)}(  B_2, C_2 ).
\end{eqnarray*}
 Thus $ w_{N,e}( . , . ) $ is sub additive and so $ w_{(N,e)}( . , . ) $  is a norm on $\mathbb{B}^2(\mathscr{H})$.
\end{proof}

In the next result we proof some properties of the norm $w_{(N,e)}(. , .)$.

\begin{proposition}\label{prop1}
Let  $ B,C\in\mathbb{B}(\mathscr{H}).$  Then \\
(a) $ w_{(N,e)}(B ,C )=\frac{1}{\sqrt{2}} w_{(N,e)}(B+C ,B-C ).$\\
 (b)$w_{(N,e)}(\Re B, \Im B)=\frac{1}{\sqrt{2}}w_{(N,e)}(B,B^*)=w_N(B).$\\
 (c)$w_{(N,e)}(B,B)=\sqrt{2}w_N(B).$\\
(d) The norm $w_{(N,e)}( . , . )$ is self adjoint.\\
(e) If the norm $N( . )$ is weakly unitarily invariant, then so is $w_{(N,e)}( . , . )$.\\
(f)  $w_{(N,e)}( B,C )=\underset{|\lambda_1|^2+\lambda_2|^2\leq1}\sup \underset{\alpha^2+\beta^2=1}{\underset{\alpha,\beta\in\mathbb{R}}\sup} N(\alpha\Re (\lambda_1B+\lambda_2C)+\beta\Im(\lambda_1B+\lambda_2C)).$ 
\end{proposition}
 \begin{proof}
 $(a)$
 It follows from Definition \ref{def1} and (\ref{eqn11}) we have \begin{eqnarray}\label{eqn12}
  w_{(N,e)}(B ,C )=\underset{|\lambda_1|^2+\lambda_2|^2\leq1}\sup w_N(\lambda_1B+\lambda_2C).
 \end{eqnarray}
 It follows form (\ref{eqn12}), \begin{eqnarray*}
  w_{(N,e)}(B+C ,B-C)&=& \underset{|\lambda_1|^2+\lambda_2|^2\leq1}\sup
 w_N(B(\lambda_1+\lambda_2)+C(\lambda_1-\lambda_2))\\&=&\sqrt{2}\underset{|\lambda_1|^2+\lambda_2|^2\leq1}\sup w_N(\lambda_1B+\lambda_2C).\\
 &=&\sqrt{2} w_{(N,e)}(B ,C ).
 \end{eqnarray*}
 $(b)$ Again from (\ref{eqn12}), we have 
 \begin{eqnarray*}
  w_{(N,e)}(\Re B ,\Im B )&=&\underset{|\lambda_1|^2+\lambda_2|^2\leq1}\sup w_N(\lambda_1 \Re B+\lambda_2 \Im B )\\
&=&\underset{|\lambda_1|^2+\lambda_2|^2\leq1}\sup \frac12 w_N(\lambda_1(B+B^*)+\lambda_2(-iB+iB^*))\\
&=&\underset{|\lambda_1|^2+\lambda_2|^2\leq1}\sup \frac12 w_N(B(\lambda_1-i\lambda_2)+B^*(\lambda_1+i\lambda_2)\\
&=&\underset{|\lambda_1|^2+\lambda_2|^2\leq1}\sup \frac{1}{\sqrt{2}} w_N(\lambda_1 B+ \lambda_2 B^*)\\
&=&\frac{1}{\sqrt{2}}w_{(N,e)}(B,B^*).
  \end{eqnarray*}
  Now, \begin{eqnarray}
  w_{(N,e)}(B,B^*)&=&\underset{|\lambda_1|^2+\lambda_2|^2\leq1}\sup \underset{\theta\in\mathbb{R}}\sup      N(\Re( e^{i\theta}(\lambda_1B+\lambda_2B^*)))\nonumber\\
  &=&\underset{|\lambda_1|^2+\lambda_2|^2\leq1}\sup \underset{\theta\in\mathbb{R}}\sup      N\left(\Re( e^{i\theta}(\lambda_1B))+\Re (e^{i\theta}(\bar{\lambda_2}B))\right)
 \label{eqn13} \end{eqnarray}
 Now,
 \begin{eqnarray*}
 &&\underset{|\lambda_1|^2+\lambda_2|^2\leq1}\sup \underset{\theta\in\mathbb{R}}\sup      N\left(\Re( e^{i\theta}(\lambda_1B))+\Re (e^{i\theta}(\bar{\lambda_2}B))\right)\\
&\leq& \underset{|\lambda_1|^2+\lambda_2|^2\leq1}\sup\left\lbrace\underset{\theta\in\mathbb{R}}\sup      N\left(\Re( e^{i\theta}(\lambda_1B))\right)+\underset{\theta\in\mathbb{R}}\sup      N\left(\Re( e^{i\theta}(\bar{\lambda_2}B))\right)\right\rbrace\\
 &=&\underset{|\lambda_1|^2+\lambda_2|^2\leq1}\sup\left\lbrace w_N(\lambda_1B)+w_N(\bar{\lambda_2}B)\right\rbrace\\
 &=&\underset{|\lambda_1|^2+\lambda_2|^2\leq1}\sup\left\lbrace |\lambda_1|+|\lambda_2|\right\rbrace w_N(B)\\
 &=&\sqrt{2}w_N(B).
 \end{eqnarray*}
 It follows from (\ref{eqn13}) that 
 \begin{eqnarray*}
  w_{(N,e)}(B,B^*)&\geq&\frac{1}{\sqrt{2}}\underset{\theta\in\mathbb{R}}\sup      N\left(\Re( e^{i\theta}B)+\Re (e^{i\theta}B)\right)\\
  &=&\sqrt{2}w_N(B).
 \end{eqnarray*}
 Therefore $w_{(N,e)}(B,C)=\sqrt{2}w_N(B).$ 
 
(c) \begin{eqnarray*}
w_{(N,e)}(B,B)&=&\underset{|\lambda_1|^2+\lambda_2|^2\leq1}\sup \underset{\theta\in\mathbb{R}}\sup      N(\Re( e^{i\theta}(\lambda_1B+\lambda_2B)))\\
&=&\underset{|\lambda_1|^2+\lambda_2|^2\leq1}\sup \underset{\theta\in\mathbb{R}}\sup      N(\Re( e^{i\theta}((\lambda_1+\lambda_2)B))\\
&=&\underset{|\lambda_1|^2+\lambda_2|^2\leq1}\sup w_N((\lambda_1+\lambda_2)B)\\
&=&\underset{|\lambda_1|^2+\lambda_2|^2\leq1}\sup|(\lambda_1+\lambda_2)|w_N(B)\\
&\leq&\underset{|\lambda_1|^2+\lambda_2|^2\leq1}\sup(|\lambda_1|+|\lambda_2|)|w_N(B)\\
&=&\sqrt{2}w_N(B).
\end{eqnarray*} 
 Also $w_{(N,e)}(B,B)\geq\sqrt{2}w_N(B),$ as $w_{(N,e)}(B,B)=\underset{|\lambda_1|^2+\lambda_2|^2\leq1}\sup|(\lambda_1+\lambda_2)|w_N(B).$\\
 Hence the result is proved.\\
 (d) \begin{eqnarray*}
 w_{(N,e)}(B^*,C^*)&=&\underset{|\lambda_1|^2+\lambda_2|^2\leq1}\sup \underset{\theta\in\mathbb{R}}\sup      N(\Re( e^{i\theta}(\lambda_1B^*+\lambda_2C^*)))\\
 &=&\underset{|\lambda_1|^2+\lambda_2|^2\leq1}\sup \underset{\theta\in\mathbb{R}}\sup      N(\Re( e^{-i\theta}(\bar{\lambda_1}B+\bar{\lambda_2}C)))\\
 &=& w_{(N,e)}(B,C).
 \end{eqnarray*}
 (e) Assume that $N(. )$ is weakly unitarily invariant and let $U\in\mathbb{B}(\mathscr{H})$ be unitary. Then\begin{eqnarray*}
 w_{(N,e)}(U^*(B,C)U)&=&\underset{|\lambda_1|^2+\lambda_2|^2\leq1}\sup \underset{\theta\in\mathbb{R}}\sup      N(\Re( e^{i\theta}(U^*(\lambda_1B+\lambda_2C)U)))\\
 &=&\underset{|\lambda_1|^2+\lambda_2|^2\leq1}\sup \underset{\theta\in\mathbb{R}}\sup  N(U^*\Re( e^{i\theta}(\lambda_1B+\lambda_2C))U).  
 \end{eqnarray*}
 By the consideration we have $ N(U^*\Re( e^{i\theta}(\lambda_1B+\lambda_2C))U)= N(\Re( e^{i\theta}(\lambda_1B+\lambda_2C)))$.\\
 Therefore \begin{eqnarray*}
 w_{(N,e)}(U^*(B,C)U)&=&\underset{|\lambda_1|^2+\lambda_2|^2\leq1}\sup \underset{\theta\in\mathbb{R}}\sup N(\Re( e^{i\theta}(\lambda_1B+\lambda_2C)))\\
 &=& w_{(N,e)}(B,C).
 \end{eqnarray*}
 (f) \begin{eqnarray*}
 w_{(N,e)}(B,C)&=&\underset{|\lambda_1|^2+\lambda_2|^2\leq1}\sup \underset{\theta\in\mathbb{R}}\sup      N(\Re( e^{i\theta}(\lambda_1B+\lambda_2C)))\\
&=&\underset{|\lambda_1|^2+\lambda_2|^2\leq1}\sup\underset{\theta\in\mathbb{R}}\sup \frac12     N\left( e^{i\theta}(\lambda_1B+\lambda_2 C)+e^{-i\theta}(\bar{\lambda_1}B^*+\bar{\lambda_2} C^*)\right) \\
&=&\underset{|\lambda_1|^2+\lambda_2|^2\leq1}\sup\underset{\theta\in\mathbb{R}}\sup      N\left( cos\theta\Re(\lambda_1B+\lambda_2 C)-sin\theta\Im (\lambda_1B+\lambda_2 C)\right).   
\end{eqnarray*}
Put $\alpha=cos\theta$ and $\beta=-sin\theta$ we have 
\begin{eqnarray*}
 w_{(N,e)}(B,C)&=&\underset{|\lambda_1|^2+\lambda_2|^2\leq1}\sup\underset{\theta\in\mathbb{R}}\sup      N\left( cos\theta\Re(\lambda_1B+\lambda_2 C)-sin\theta\Im (\lambda_1B+\lambda_2 C)\right)\\
 &=&\underset{|\lambda_1|^2+\lambda_2|^2\leq1}\sup \underset{\alpha^2+\beta^2=1}{\underset{\alpha,\beta\in\mathbb{R}}\sup} N(\alpha\Re (\lambda_1B+\lambda_2C)+\beta\Im((\lambda_1B+\lambda_2C))
\end{eqnarray*}
  Hence the results are proved.
 
 \end{proof}

Next we proof some bounds for generalized Euclidean operator radius of a pair of bounded linear operator.
  \begin{theorem}\label{thh11}
  Let  $ B,C\in\mathbb{B}(\mathscr{H}).$ then 

  (a) $ w_{(N,e)}(B ,C )\geq\max\left\lbrace w_N(B), w_N(C)\right\rbrace. $\\
  (b)  $ w_{(N,e)}(B ,C )\leq\sqrt{w_N^2(B)+w_N^2(C)}.$
 
  \end{theorem}
  \begin{proof}
  We have $$ w_{(N,e)}(B,C)=\underset{|\lambda_1|^2+\lambda_2|^2\leq1}\sup \underset{\theta\in\mathbb{R}}\sup      N\left(\Re \left(e^{i\theta}(\lambda_1B+\lambda_2C)\right)\right),$$ 
  for every   $\lambda_1,\lambda_2\in\mathbb{C}.$ If we take $\lambda_1=1$ and $\lambda_2=0$ we have that 
  $$ w_{(N,e)}(B,C)\geq \underset{\theta\in\mathbb{R}}\sup      N(\Re (e^{i\theta}B))=w_N(B).$$ 
  In similar way, if we take $\lambda_1=0$ and $\lambda_2=1$ we get  $$ w_{(N,e)}(B,C)\geq \underset{\theta\in\mathbb{R}}\sup      N(\Re( e^{i\theta}C))=w_N(C).$$ Hence the first result is proved.\\
  Now,
  \begin{eqnarray*}
   w_{(N,e)}(B,C)&=&\underset{|\lambda_1|^2+\lambda_2|^2\leq1}\sup \underset{\theta\in\mathbb{R}}\sup     N\left(\Re \left(e^{i\theta}(\lambda_1B+\lambda_2C)\right)\right)\\
   &=&\underset{|\lambda_1|^2+\lambda_2|^2\leq1}\sup \underset{\theta\in\mathbb{R}}\sup      N(\Re( e^{i\theta}(\lambda_1B))+\Re( e^{i\theta}(\lambda_2C)))\\
   &\leq& \underset{|\lambda_1|^2+\lambda_2|^2\leq1}\sup\left\lbrace \underset{\theta\in\mathbb{R}}\sup      N(\Re (e^{i\theta}(\lambda_1B)))+\underset{\theta\in\mathbb{R}}\sup      N(\Re( e^{i\theta}(\lambda_2C)))\right\rbrace\\
   &=&\underset{|\lambda_1|^2+\lambda_2|^2\leq1}\sup\left\lbrace w_N(\lambda_1B)+ w_N(\lambda_2C)\right\rbrace\\
   &=&\underset{|\lambda_1|^2+\lambda_2|^2\leq1}\sup\left\lbrace |\lambda_1|w_N(B)+ |\lambda_2| w_N(C)\right\rbrace\\
   &=&\sqrt{w_N^2(B)+w_N^2(C)}.
   \end{eqnarray*}
    Hence the theorem is proved.  	
  \end{proof}
  It is easy to proof from the definition of $w_N( . )$ that $w_N(T)=N(T)$, if $T\in \mathbb{B}(\mathscr{H})$ is self adjoint.\\
  The following corollary is an immediate consequence of the Theorem \ref{thh11}.
  \begin{cor}
   Let  $ B,C\in\mathbb{B}(\mathscr{H})$ be any two self adjoint operator. Then $$  \max\left\lbrace N(B), N(C)\right\rbrace \leq w_{(N,e)}(B ,C )\leq\sqrt{N^2(B)+N^2(C)}.$$
  \end{cor}

\begin{remark}
If we replace B,C by $\Re T,$ $\Im T$ respectively, for any $T\in\mathbb{B}(\mathscr{H})$ in the above Theorem \ref{thh11}, we have $$ \max\left\lbrace N(\Re T), N(\Im T)\right\rbrace\leq w_N(T)\leq\sqrt{N^2(\Re T)+N^2(\Im T)}.$$
\end{remark}
Next theorem state as follows:
\begin{theorem}
  Let  $ B,C\in\mathbb{B}(\mathscr{H}).$ then 

  (a) $\frac{1}{\sqrt{2}}\max\left\lbrace   w_N(B+C), w_N(B-C)\right\rbrace\leq w_{(N,e)}(B ,C )\leq 
\frac{1}{\sqrt{2}}  \sqrt{w_N^2(B+C)+w_N^2(B-C)}.$
 
  \label{thh12}\end{theorem}
\begin{proof}
 We have $$ w_{(N,e)}(B,C)=\underset{|\lambda_1|^2+\lambda_2|^2\leq1}\sup \underset{\theta\in\mathbb{R}}\sup     N\left(\Re \left(e^{i\theta}(\lambda_1B+\lambda_2C)\right)\right),$$ 
  for every   $\lambda_1,\lambda_2\in\mathbb{C}.$ Consider $\lambda_1=\frac{1}{\sqrt{2}}$ and $\lambda_2=\frac{1}{\sqrt{2}}$ we have that 
  $$ w_{(N,e)}(B,C)\geq \underset{\theta\in\mathbb{R}}\sup      N\left(\Re\left( e^{i\theta}\left(\frac{1}{\sqrt{2}}(B+C)\right)\right)\right)=\frac{1}{\sqrt{2}}w_N(B+C).$$ 
  Consider $\lambda_1=\frac{1}{\sqrt{2}}$ and $\lambda_2=\frac{-1}{\sqrt{2}}$ we have that 
  $$ w_{(N,e)}(B,C)\geq \underset{\theta\in\mathbb{R}}\sup      N\left(\Re\left( e^{i\theta}\left(\frac{1}{\sqrt{2}}(B-C)\right)\right)\right)=\frac{1}{\sqrt{2}}w_N(B-
  C).$$  Hence the first result is proved.\\
 Replacing $B$, $C$ by $B+C$ and $B-C$ respectively in Theorem \ref{thh11}(b) and using Proposition \ref{prop1} we get desire second inequality.
\end{proof}

\begin{cor}
   Let  $ B,C\in\mathbb{B}(\mathscr{H})$ be any two self adjoint operator. Then $$\frac{1}{\sqrt{2}}\max\left\lbrace   N(B+C), N(B-C)\right\rbrace\leq w_{(N,e)}(B ,C )\leq 
\frac{1}{\sqrt{2}}  \sqrt{N^2(B+C)+N^2(B-C)}.$$
  \end{cor}
  
  Now, if we consider $B= T$ and $C= T^*$ in Theorem \ref{thh12}, we get the following inequality.
\begin{cor}\label{cor1}
	Let $T\in \mathbb{B}(\mathscr{H}),$ then
 	\begin{eqnarray*}
		 \frac{1}{2}\max\left\lbrace   w_N(T+T^*), w_N(T-T^*)\right\rbrace\leq w_N(T) )\leq 
\frac{1}{2}  \sqrt{w_N^2(T+T^*)+w_N^2(T-T^*)}.
	\end{eqnarray*}
\end{cor}

\begin{remark}
If we replace B,C by $\Re T,$ $\Im T$ respectively, for any $T\in\mathbb{B}(\mathscr{H})$ in the above Theorem \ref{thh12}, we have $$\frac{1}{\sqrt{2}}\max\left\lbrace   N(\Re T+\Im T), N(\Re T-\Im T)\right\rbrace\leq w_N(T )\leq 
\frac{1}{\sqrt{2}}  \sqrt{N^2(\Re T+\Im T)+N^2(\Re T-\Im T)}.$$
\end{remark}

Next lower bound for $w_{(N,e)}(B,C)$ reads as follows.

\begin{theorem}

Let $ B,C\in\mathbb{B}(\mathscr{H})$, then $$\frac{1}{2}w_N(B+e^{i\theta }C)+\frac12|w_N(B)-w_N(C)|\leq w_{(N,e)}(B,C),$$ holds for all $\theta\in\mathbb{R}.$
\label{thh13}\end{theorem}
\begin{proof}

It follows from Theorem \ref{thh11} that, \begin{eqnarray*}
w_{(N,e)}(B,C)&\geq&\max\left\lbrace w_N(B), w_N(C)\right\rbrace\\
&=&\frac12\left(w_N(B)+ w_N(C)\right)+\frac12|w_N(B)-w_N(C)|\\
&\geq&\frac12 w_N(B+e^{i\theta} C)+\frac12|w_N(B)-w_N(C)|
\end{eqnarray*}
\end{proof}

\begin{remark}
(i) Clearly, from the bound in Theorem \ref{thh13} we say that if $w_{(N,e)}(B,C)=\frac12 w_N(B+e^{i\theta} C)$ then $w_N(B)=w_N(C).$ By considering $C=B$ we conclude that converse is not true.\\
(ii) Replacing $B$ by $\Re(T)$ and $C$ by $\Im(T)$ in  Theorem \ref{thh13} we get the following lower bound for the numerical radius of $T\in \mathbb{B}(\mathscr{H})$:
\begin{eqnarray*}
 w_N(T) \geq 
 \frac{1}{2}N(\Re(T)+e^{i\theta}\Im(T))+\frac12|N(\Re T)-N(\Im T))|. 
\end{eqnarray*}
(iii)Replacing $B$ by $T$ and $C$ by $T^*$ in  Theorem \ref{thh13} we get the following lower bound for the numerical radius of $T\in \mathbb{B}(\mathscr{H})$:\begin{eqnarray}
w_N(T) \geq\frac{1}{2\sqrt{2}}w_N(T+e^{i\theta}T^*).\label{eqn14}\end{eqnarray} 
\end{remark}

We next obtain the following inequality.\\
To proof our next theorem we need following definition.\\
\begin{definition}
A norm $N( . )$ on $\mathbb{B}(\mathscr{H})$ is an algebra norm if
$$ N (AB) \leq N (A)N (B),$$ for every $A,B\in\mathbb{B}(\mathscr{H}).$
\end{definition}

\begin{theorem}\label{thh15}
Let $N( . )$ is a algebra norm, self adjoint and $ B,C\in\mathbb{B}(\mathscr{H})$, then $$\frac18N(C^*C+B^*B)+\frac12\max\left\lbrace w_N(B),w_N(C)\right\rbrace|w_N(B+C)-w_N(B-C)|\leq w_{(N,e)}^2(B,C).$$
\end{theorem}
\begin{proof}
It follows from Theorem \ref{thh12} that
\begin{eqnarray*}
w_{(N,e)}^2(B ,C )&\geq&\frac{1}{2}\max\left\lbrace   w_N^2(B+C), w_N^2(B-C)\right\rbrace\\
&=&\frac14\left(w_N^2(B+C)+ w_N^2(B-C)+|w_N^2(B+C)- w_N^2(B-C)|\right)\\
&\geq&\frac{1}{16}\left(N^2(B+C)+ N^2(B-C)\right)\\
&&+\frac14\left(w_N(B+C)+ w_N(B-C)\right)|w_N(B+C)- w_N(B-C)|\\
&&(\textit{ because $w_N(B)\geq\frac12 N(B)$})\\
&\geq& \frac{1}{16}\left(N((B+C)(B+C)^*)+ N((B-C)(B-C)^*)\right)\\
&&+\frac14\left(w_N(B+C)+ w_N(B-C)\right)|w_N(B+C)- w_N(B-C)|\\
&&(\textit{ because $N(.)$ is algebra norm on $\mathbb{B}(\mathscr{H})$})\\
&\geq& \frac{1}{16}\left(N((B+C)(B+C)^*)+ N((B-C)(B-C)^*)\right)\\
&&+\frac14 w_N((B+C)+(B-C)) |w_N(B+C)- w_N(B-C)|\\
&\geq&\frac{1}{16}\left(N((B+C)(B+C)^*+(B-C)(B-C)^*\right)\\
&&+\frac12w_N(B)|w_N(B+C)- w_N(B-C)|\\
&=&\frac{1}{8}N(BB^*+CC^*)+\frac12w_N(B)|w_N(B+C)- w_N(B-C)|.
\end{eqnarray*}
Therefore, \begin{eqnarray}\label{eqn15}
w_{(N,e)}^2(B ,C )\geq\frac{1}{8}N(BB^*+CC^*)+\frac12w_N(B)|w_N(B+C)- w_N(B-C)|.\end{eqnarray}
Now interchanging $B$ and $C$ in (\ref{eqn15}), we have that\begin{eqnarray}\label{eqn16}
w_{(N,e)}^2(B ,C )\geq\frac{1}{8}N(BB^*+CC^*)+\frac12w_N(C)|w_N(B+C)- w_N(B-C)|.\end{eqnarray}  
Therefore, the desire inequality follows by combining the inequalities in (\ref{eqn15}) and (\ref{eqn16}).
\end{proof}

Following Theorem \ref{thh15}, $w_{(N,e)}^2(B,C)=\frac18N(BB^*+CC^*)$ implies $w_N(B+C)=w_N(B-C).$ But, by considering $C=0$, we conclude that the converse part is not true.
The following corollary is an immediate consequence of Theorem \ref{thh15} assuming $B$ and $C$ to be self adjoint operators.
\begin{cor}
Let $N( . )$ is a algebra norm, self adjoint and $ B,C\in\mathbb{B}(\mathscr{H})$ be self adjoint, then $$\frac18N(C^2+B^2)+\frac12\max\left\lbrace N(B),N(C)\right\rbrace|N(B+C)-N(B-C)|\leq w_{(N,e)}^2(B,C).$$
\label{cor12}\end{cor}

In particular, by considering $B=\Re T$ and $C=\Im T$ in Corollary \ref{cor12} we obtain the following upper bounds for $w_N(T).$

\begin{cor}
Let $N( . )$ is a algebra norm, self adjoint and $ T\in\mathbb{B}(\mathscr{H})$, then $$\frac{1}{16} N(T^*T+TT^*)+\frac12\max\left\lbrace N(\Re T),N(\Im T)\right\rbrace|N(\Re T + \Im T)-N(\Re T-\Im T)|\leq w_N^2(T).$$
\end{cor}
Now, if we consider $B=T$ and $C=T^*$ in Theorem \ref{thh15}, we get the following inequality.
\begin{cor}
Let $N( . )$ is a algebra norm, self adjoint and $ T\in\mathbb{B}(\mathscr{H})$, then $$\frac{1}{16} N(T^*T+TT^*)+\frac12w_N(T)|N(\Re T)-N(\Im T)|\leq w_N^2(T).$$
\end{cor}

Replacing $B$ by $B+C$ and $C$ by $B-C$ in Theorem \ref{thh15} and using the Proposition \ref{prop1} we have the following bounds for $w_{(N,e)}(B,C)$.
\begin{cor}
Let $N( . )$ is a algebra norm, self adjoint and $ B,C\in\mathbb{B}(\mathscr{H})$, then $$\frac18N(C^*C+B^*B)+\frac12\max\left\lbrace w_N(B+C),w_N(B-C)\right\rbrace|w_N(B)-w_N(C)|\leq w_{(N,e)}^2(B,C).$$
\end{cor}

\section{ Hilbert-Schmidt Euclidean operator radius inequalities}

In this section, we study the norm $w_{(N,e)}(  . ,  . )$ when $N( . )$ is the Hilbert-Schmidt norm. Recall that an operator $T\in\mathbb{B}(\mathscr{H})$ is said to belong to the  Hilbert-Schmidt class $C_2$ if $\sum_{i,j=1}^{\infty}|\langle Te_i,e_j\rangle|^2=\sum_{i=1}^{\infty}\|Te_i\|^2$ is finite for some (hence, for any) orthonormal basis $\left\lbrace e_i\right\rbrace_{1}^{\infty}$.  For $T \in C_2$, let $\|A\|_2=\left(\sum_{i=1}^{\infty}\|Te_i\|^2\right)^\frac12$ be the Hilbert-Schmidt norm of $T$. Note that for $T \in C_2$, $\|T\|_{2}^{2}=tr (T^*T)$. \\
When $N( . )$ is the Hilbert-Schmidt norm $\| \|_2$, the norm $w_N( . )$ is denote as $w_2( .)$ and the norm $w_{(N,e)}( . , . )$ is denoted by $w_{(2,e)}( . , . )$ and is defined by $$w_{(2,e)}( B, C )=\underset{|\lambda_1|^2+\lambda_2|^2\leq1}\sup \underset{\theta\in\mathbb{R}}\sup      \|\Re( e^{i\theta}(\lambda_1B+\lambda_2C))\|_2,$$ for every $B,C\in\mathbb{B}(\mathscr{H})$ and called it   Hilbert-Schmidt Euclidean operator radius. 

Now we proof some bounds of  Hilbert-Schmidt Euclidean operator radius. 

\begin{theorem}
Let $B,C\in\mathbb{B}(\mathscr{H})$. Then \begin{eqnarray*}
 w^2_{(2,e)}( B, C )&\geq&\frac14\left( |tr(B^2)+tr(C^2)+2tr(BC)|\right)\\
 &&+\frac14\left(\|B\|_{2}^{2}+\|C\|_{2}^{2}+ 2Re( tr(BC^*))\right),\end{eqnarray*} where $ Re(z)$ is the real part of $z\in\mathbb{C}.$
\label{thh16}\end{theorem} 
\begin{proof}
\begin{eqnarray*}
&&w^2_{2,e}( B, C )\\
&=&\underset{|\lambda_1|^2+\lambda_2|^2\leq1}\sup \underset{\theta\in\mathbb{R}}\sup      \|\Re e^{i\theta}(\lambda_1B+\lambda_2C)\|_{2}^{2}\\
&=&\underset{|\lambda_1|^2+\lambda_2|^2\leq1}\sup \underset{\theta\in\mathbb{R}}\sup\hspace{0.2cm}   tr\left(\Re( e^{i\theta}(\lambda_1B+\lambda_2C))^2\right)\\
&=&\underset{|\lambda_1|^2+\lambda_2|^2\leq1}\sup \underset{\theta\in\mathbb{R}}\sup\hspace{0.2cm} tr [ \frac12\Re(e^{2i\theta}(\lambda_1B+\lambda_2C)^2)\\
&&+\frac14\left(|\lambda_1|^2BB^*+|\lambda_1|^2 B^*B+|\lambda_2|^2 CC^*+|\lambda_2|^2 C^*C\right)\\
&&+\frac14\left(2\Re (\lambda_1\bar{\lambda_2} BC^*)+2\Re (\lambda_1\bar{\lambda_2} C^*B)\right)]\\
&=&\underset{|\lambda_1|^2+\lambda_2|^2\leq1}\sup \underset{\theta\in\mathbb{R}}\sup  \frac12\Re(e^{2i\theta}tr(\lambda_1B+\lambda_2C)^2)\\&&+\frac12\left(|\lambda_1|^2\|B\|_{2}^{2}+|\lambda_2|^2\|C\|_{2}^{2}\right)+  \frac12tr(\Re(\lambda_1\bar{\lambda_2}
BC^*)+\frac12tr(\Re(\lambda_1\bar{\lambda_2}
C^*B)\\
&=&\underset{|\lambda_1|^2+\lambda_2|^2\leq1}\sup   \frac12|tr(\lambda_1B+\lambda_2C)^2|\\&&+\frac12\left(|\lambda_1|^2\|B\|_{2}^{2}+|\lambda_2|^2\|C\|_{2}^{2}\right)+  \frac12tr(\Re(\lambda_1\bar{\lambda_2}
BC^*)+\frac12tr(\Re(\lambda_1\bar{\lambda_2}
C^*B)\\
&\geq& \frac14\left( |tr((B+C)^2)|+\|B\|_{2}^{2}+\|C\|_{2}^{2}+ Re( tr(BC^*))+Re( tr(C^*B))\right),\\
&&\textit{( where $ Re(z)$ is the real part of $z\in\mathbb{C}$)}.\\
&=&\frac14\left( |tr(B^2)+tr(C^2)+2tr(BC)|+\|B\|_{2}^{2}+\|C\|_{2}^{2}+2 Re( tr(BC^*))\right)
\end{eqnarray*}
\end{proof}

 The following corollaries are the immediate consequence of the above Theorem \ref{thh16}. 
 
 Replacing $B$ by $B+C$ and $C$ by $B-C$ in Theorem \ref{thh16} and using the Proposition \ref{prop1} we have the following bounds for $w_{2,e}(B,C)$.
\begin{cor}
Let $B,C\in\mathbb{B}(\mathscr{H})$. Then \begin{eqnarray*}
 w^2_{(2,e)}( B, C )&\geq&\frac18\left( |tr((B+C)^2)+tr((B-C)^2)+2tr((B+C)(B-C))|\right)\\
 &&+\frac18\left(\|B+C\|_{2}^{2}+\|B-C\|_{2}^{2}\right)\\
 &&+\frac14 Re( tr((B+C)(B-C)^*)),\end{eqnarray*} where $ Re(z)$ is the real part of $z\in\mathbb{C}.$
\end{cor}
Now, if we consider $B=T$ and $C=T$ in Theorem \ref{thh16}, and using the Proposition \ref{prop1} we have the following bounds for $w_2(T)$.
\begin{cor}
Let $T\in\mathbb{B}(\mathscr{H})$. Then
 $$ w^2_2(T)\geq\frac12 \|\ T\|^2_2+\frac12|tr(T^2)|.$$
\label{cor31}\end{cor}

\begin{theorem}
Let $B,C\in\mathbb{B}(\mathscr{H})$. Then $$w^2_{(2,e)}( B, C )\leq\frac12\left(\max\left\lbrace |tr(B^2)|,|tr(C^2)|\right\rbrace+|tr(BC)|
+\max\left\lbrace\|B\|_{2}^{2},\|C\|_{2}^{2}\right\rbrace
+ |tr(
BC^*)|\right).$$ 
\label{thh17}\end{theorem} 
\begin{proof}
It follows from Theorem \ref{thh16} that
\begin{eqnarray*}
&&w^2_{(2,e)}( B, C )\\
&=&\underset{|\lambda_1|^2+\lambda_2|^2\leq1}\sup[\frac12 |tr(\lambda_1B+\lambda_2C)^2|+\frac12\left(|\lambda_1|^2\|B\|_{2}^{2}+|\lambda_2|^2\|C\|_{2}^{2}\right)\\
&&+  \frac12tr(\Re(\lambda_1\bar{\lambda_2}
BC^*)+\frac12tr(\Re(\lambda_1\bar{\lambda_2}
C^*B)]\\
&\leq&\underset{|\lambda_1|^2+\lambda_2|^2\leq1}\sup[\frac12 |tr(\lambda_1B+\lambda_2C)^2|+\frac12\left(|\lambda_1|^2+|\lambda_2|^2\right)\max\left(\|B\|_{2}^{2},\|C\|_{2}^{2}\right)\\
&&+  \frac12Re(\lambda_1\bar{\lambda_2}tr(
BC^*)+\frac12 Re(\lambda_1\bar{\lambda_2}
tr(C^*B))]\\
&\leq&\underset{|\lambda_1|^2+\lambda_2|^2\leq1}\sup[\frac12|tr(\lambda_1B+\lambda_2C)^2|+\frac12\max\left(\|B\|_{2}^{2},\|C\|_{2}^{2}\right)\\
&&+  \frac12|\lambda_1||\lambda_2|\left(|tr(
BC^*)|+|tr(C^*B)|\right)]\\
&\leq&\frac12\underset{|\lambda_1|^2+\lambda_2|^2\leq1}\sup |tr(\lambda_1B+\lambda_2C)^2|+\frac12\max\left(\|B\|_{2}^{2},\|C\|_{2}^{2}\right)\\
&&+  \frac12\underset{|\lambda_1|^2+\lambda_2|^2\leq1}\sup(|\lambda_1||\lambda_2|)\left(|tr(
BC^*)|+|tr(C^*B)|\right)\\
&=&\frac12\underset{|\lambda_1|^2+\lambda_2|^2\leq1}\sup|tr(\lambda_1B+\lambda_2C)^2|+\frac12\max\left(\|B\|_{2}^{2},\|C\|_{2}^{2}\right)
+  \frac12|tr(
BC^*)|.\\
&\leq&\frac12\underset{|\lambda_1|^2+\lambda_2|^2\leq1}\sup|tr((\lambda_1)^2B^2)+tr((\lambda_2)^2C^2)+2\lambda_1\lambda_2tr(BC)|\\
&&+\frac12\max\left(\|B\|_{2}^{2},\|C\|_{2}^{2}\right)
+  \frac12|tr(
BC^*)|.\\
&\leq&\frac12\underset{|\lambda_1|^2+\lambda_2|^2\leq1}\sup\left(\max\left\lbrace |tr(B^2)|,|tr(C^2)|\right\rbrace+2|\lambda_1\lambda_2||tr(BC)|\right)\\
&&+\frac12\max\left(\|B\|_{2}^{2},\|C\|_{2}^{2}\right)
+  \frac12|tr(
BC^*)|.\\
&=&\frac12\left(\max\left\lbrace |tr(B^2)|,|tr(C^2)|\right\rbrace+|tr(BC)|
+\max\left\lbrace\|B\|_{2}^{2},\|C\|_{2}^{2}\right\rbrace
+ |tr(
BC^*)|\right).
\end{eqnarray*}
\end{proof}
The following corollaries are the immidiate consequence of the above Theorem \ref{thh17}. 
 
Now, if we consider $B=T$ and $C=T$ in Theorem \ref{thh17}, and using the Proposition \ref{prop1} we have the following bounds for $w_2(T)$.
\begin{cor}
Let $T\in\mathbb{B}(\mathscr{H})$. Then
 $$ w^2_2(T)\leq\frac12 \|\ T\|^2_2+\frac12|tr(T^2)|.$$
\label{cor32}\end{cor}
It follows from Corollary \ref{cor31} and Corollary \ref{cor32} that  $$ w^2_2(T)=\frac12 \|\ T\|^2_2+\frac12|tr(T^2)|,$$ which is also proved in \cite{AK2019}.\\

	\noindent \textbf{Competing Interests.}\\
On behalf of all authors, the corresponding author  declares that there is no financial or non-financial interests that are directly or indirectly related to the work submitted for publication.

\bibliographystyle{amsplain}

\end{document}